\documentclass[conference]{IEEEtran}

\usepackage{graphicx}
\usepackage{amsmath}
\usepackage[margin=.75in]{geometry}
\ifCLASSINFOpdf

\else

\fi
\begin{document}

\title{\vspace{.25in}Theorems about Ergodicity and Class-Ergodicity of Chains with Applications in Known Consensus Models}

\author{\IEEEauthorblockN{Sadegh Bolouki}
\IEEEauthorblockA{GERAD and \'Ecole Polytechnique de Montr\'eal\\
Montreal, Quebec H3C 3A7\\
Email: sadegh.bolouki@polymtl.ca}
\and
\IEEEauthorblockN{Roland P. Malham\'e}
\IEEEauthorblockA{GERAD and \'Ecole Polytechnique de Montr\'eal\\
Montreal, Quebec H3C 3A7\\
Email: roland.malhame@polymtl.ca}}

\maketitle

\begin{abstract}
    In a multi-agent system, unconditional (multiple) consensus is the property of reaching to (multiple) consensus irrespective of the instant and values at which states are initialized. For linear algorithms, occurrence of unconditional (multiple) consensus turns out to be equivalent to (class-)ergodicity of the transition chain $(A_n)$. For a wide class of chains, chains with so-called balanced asymmetry property, necessary and sufficient conditions for ergodicity and class-ergodicity are derived. The results are employed to analyze the limiting behavior of agents' states in the JLM model, the Krause model, and the Cucker-Smale model. In particular, unconditional single or multiple consensus occurs in all three models. Moreover, a necessary and sufficient condition for unconditional consensus in the JLM model and a sufficient condition for consensus in the Cucker-Smale model are obtained.
\end{abstract}


\IEEEpeerreviewmaketitle


\newtheorem{theorem}{Theorem}
\newtheorem{definition}{Definition}
\newtheorem{lemma}{Lemma}
\newtheorem{corollary}{Corollary}
\newtheorem{remark}{Remark}
\newtheorem{note}{Note}

\section{Introduction}

Linear consensus algorithms in multi-agent systems have gained an increasing attention in various fields of research. A general linear consensus protocol is described by the following equation:
\begin{equation}
  X(n+1) = A_n X(n), \,\,\, n \geq 0,
  \label{modelss}
\end{equation}
where $X(n)$ stands for the vector of states and $(A_n)$, the \textit{transition chain}, is a sequence of row stochastic matrices, called \textit{transition matrices}. A matrix is called row stochastic if each of its rows sums to 1. The transition matrices may be known a priory (exogenous models) or may depend on agents' states (endogenous models).

A simple exogenous model, when the transition chain consists of identical row stochastic matrices, was first introduced by DeGroot in \cite{DeGroot:74}. Later, the same consensus problem, but with time dependent transition matrices, was considered in \cite{Chatterjee:77}. The authors related the consensus problem to ergodicity of chain $(A_n)$ and studied backward products of row stochastic matrices to analyze ergodicity of the chain. The results of \cite{Chatterjee:77} were improved in \cite{Tsit:86,Tsit:84,Tsit:89a}, as weaker sufficient conditions for consensus were derived. Sufficient conditions obtained in \cite{Tsit:86,Tsit:84,Tsit:89a} can be expressed briefly in two important conditions; non vanishing interaction rates and repeated connectivity of communication graph. Alternatively, Viscek et al. in \cite{Vicsek:95} modeled a system of agents moving on a plane with the same speed but different directions. In the Viscek model, the parameter to be updated is the heading of agents. Later, by linearizing the Viscek model, Jadbabaie et al. \cite{Jadba:03} obtained sufficient conditions for consensus to occur. In the JLM model, each agent updates its state by taking the arithmetic mean of the current states of its neighbors and itself. It is worth mentioning that the JLM model is an exogenous model, since for each agent $i$, the set of $i$'s neighbors at any time is pre-defined. Many authors have considered the JLM model and used various techniques to weaken the sufficient conditions for consensus as much as possible \cite{Tsit:05,Moreau:05,Hend:08,Hend:06,Li:04}. Together with the notion of consensus, existence of individual limits for agents' states, i.e., multiple consensus, has been studied in some articles such as \cite{Lorenz:05,Bolouki:11,Hend:11}). Moreover, necessary conditions for consensus were investigated by Touri and Nedi\'c \cite{Touri:10a,Touri:10b,Touri:11}. In the recent work \cite{Bolouki:12}, the authors obtained necessary and sufficient conditions for consensus and multiple consensus for a broad class of consensus algorithms.

On the other hand, the limiting behavior of agents in endogenous model has been widely studied in the literature. The Krause model \cite{Krause:97} is a well-known example of endogenous averaging algorithms. As in the JLM model, each agents modifies its states to the arithmetic mean of its neighbors' states and its own state. However, unlike the JLM model, neighbors of an agent is not pre-defined. More precisely, two agents are called neighbors at time $n \geq 0$ if their states differ by a number less than or equal to some pre-specified constant $R$. In \cite{Krause:97}, the limiting behavior of agents' states was studied and it was proved that agents eventually separate to disjoint clusters with inter-cluster distances greater than $R$. The Krause model was further studied in \cite{Blondel:07,Blondel:09}, and the results were improved.

Another interesting endogenous linear model was introduced in \cite{Cucker:07}, known as the Cucker-Smale model or the C-S model, in which birds are considered as agents seeking to fly with the same velocity in the space. In the C-S model, velocities are updated according to an averaging algorithm, where birds' interaction rates are functions of distance between birds. In \cite{Cucker:07}, sufficient condition for converging to the same velocity were derived.

Our aim is to investigate the limiting behavior of agents in various linear averaging models. This paper is organized as follows. Section II explicitly describes important notions and terminology required for the analysis of limiting behavior. The main theorems about consensus and multiple consensus  are presented in Section III. By considering well-known models in Section IV, applications of the main results are illustrated. Concluding remarks in Section V close the paper.

\subsection{Notations}

\label{notations}

  Through this article,

  \begin{itemize}

  \item $S$ is the set of agents and $s = |S|$ is the number of agents.

  \item $n$ stands for discrete time index.

  \item $X(n) = [X_1(n)\,\, \cdots \,\, X_s(n)]^T$, $n \geq 0$, is the state vector.

%

  \item Unless otherwise specified, a chain will designate a sequence of row stochastic matrices.

  \item $A_n$, $n \geq 0$ is the matrix of interaction rates $a_{ij}(n)$, $1 \leq i,j \leq s$.


  \end{itemize}


\section{Notions and Terminology}

\subsection{Ergodicity and Class-Ergodicity}

\begin{definition}
  Consider a multi-agent system with updates equation described in Eq. (\ref{modelss}). Unconditional consensus in system Eq. (\ref{modelss}) is defined as occurrence of consensus, no matter at what instant or values states are initialized.
\end{definition}
We now recall the definition of ergodicity from \cite{Touri:10a}. Let $(A_n)$ be a chain of row stochastic matrices. For $k > l \geq 0$, let $A(n,k) \equiv A_{n-1}A_{n-2}\ldots A_k$. Now,
\begin{definition}
  A chain $(A_n)$ of row stochastic matrices is called \textit{ergodic} if for each $k \geq 0$, $\lim_{n \rightarrow \infty} A(n,k)$ exists and is equal to a matrix with identical rows.
\end{definition}
One can show that occurrence of unconditional consensus in system Eq. (\ref{modelss}) is equivalent to ergodicity of chain $(A_n)$. This describes the relation of unconditional consensus and ergodicity. In the following, we define multiple consensus, which is another important notion in multi-agent systems.
\begin{definition}
  Consider the system with dynamics described by Eq. (\ref{modelss}). By unconditional multiple consensus in system Eq. (\ref{modelss}), we mean that for every $i$, $1 \leq i \leq s$, $\lim_{n \rightarrow \infty} X_i(n)$ exists, no matter at what instant or values states are initialized.
\end{definition}
To formulate multiple consensus as a property of chains of row stochastic matrices, we shall introduce class-ergodicity:
\begin{definition}
  A chain $(A_n)$ of row stochastic matrices is said to be \textit{class-ergodic} if for every $k \geq 0$, $\lim_{n \rightarrow \infty} A(n,k)$ exists and can be relabeled as a block diagonal matrix with constant rows.
\end{definition}
Clearly, Class-ergodic of chain $(A_n)$ implies unconditional multiple consensus in system Eq. (\ref{modelss}). The converse is true also \cite{Bolouki:12}. Therefore, occurrence of unconditional multiple consensus in a system with dynamics described by Eq. (\ref{modelss}) is equivalent to class-ergodicity of chain $(A_n)$.

We shall investigate ergodicity and class-ergodicity of chains. We now provide essential notions that are employed to obtain our main theorems.

\subsection{$l_1$-approximation \cite{Touri:10b}}

\begin{definition}
  Chain $(A_n)$ is said to be $l_1$-approximation of chain $(B_n)$ if
  \begin{equation}
    \sum_{n=0}^{\infty} \| A_n-B_n \| < \infty,
    \label{al-equal}
  \end{equation}
  with respect to the \textit{max norm} which is equal to the maximum of the absolute values of the matrix entries.
\end{definition}

It is not difficult to show that $l_1$-approximation is an equivalence relation in the set of chains of stochastic matrices. Importance of the $l_1$-approximation notion comes from the following lemma that is a result of Lemma 1, stated and proved in \cite{Touri:10b}.
\begin{lemma}
  Let $(A_n)$ be $l_1$-approximation of chain $(B_n)$. Then, $(A_n)$ is class-ergodic if and only if $(B_n)$ is.
  \label{partial ergodicity}
\end{lemma}

\subsection{Infinite Flow and Absolute Infinite Flow}

According to \cite{Touri:10a}, chain $(A_n)$ has the infinite flow property if for every $T \subset S = \{ 1,\ldots,s \}$, excluding the empty set and $S$ the following holds:
\begin{equation}
 \sum_{n = 0}^{\infty} \sum_{i \in T, j \not\in T} (a_{ij}(n) + a_{ji}(n)) = \infty.
 \label{infinite-flow}
\end{equation}

The absolute infinite flow property of a chain is defined as follows \cite{Touri:11}.
\begin{definition}
  A chain $(A_n)$ of row stochastic matrices has the absolute infinite flow property if
  \begin{equation}
    \begin{array}{ll}
       \sum_{n=0}^{\infty} \Big( \sum_{i \in T(n+1)} \sum_{j \in \bar{T}(n)} a_{ij}(n)\\
       \hspace{.5in}+ \sum_{i \in \bar{T}(n+1)} \sum_{j \in T(n)} a_{ij}(n)\Big) = \infty
    \label{abs}
    \end{array}
  \end{equation}
  where the $T(0),T(1),\ldots$ is an arbitrary sequence of subsets of $S$ with the same cardinality, and $\bar{T}_i$ denotes the complement of $T_i$ in $S$.
\end{definition}

Clearly, the infinite flow property is implied by the absolute infinite flow property. In \cite{Touri:10a}, it has been shown that the infinite flow property is necessary for a chain to be ergodic \cite{Touri:10a}. Following \cite{Touri:10a}, the authors later showed that the absolute infinite flow property is also necessary for ergodicity \cite{Touri:11}. In addition, they proved necessity and sufficiency of the absolute infinite property in case of chains of doubly stochastic matrices.

\subsection{Balanced Asymmetry}

We now define another property of chain of row stochastic matrices, called \textit{balanced asymmetry}, which plays a key role in our analysis.

\begin{definition}
  Let $(A_n)$ be a chain of row stochastic matrices. Chain $(A_n)$ is said to be \textit{balanced asymmetric} if there exists a finite $M \geq 1$ such that for any two non empty subsets $S_1$ and $S_2$ of $S=\{1,\ldots,s\}$ with the same cardinality, we have
  \begin{equation}
    \sum_{i \in S_1} \sum_{j \in \bar{S}_2} a_{ij}(n) \leq M \sum_{i \in \bar{S}_1} \sum_{j \in S_2} a_{ij}(n), \,\,\,\,\,\forall n \geq 0.
    \label{b-asym}
  \end{equation}
\end{definition}

\begin{remark}
  For those chains that are $l_1$-approximation of balanced asymmetric chains, the infinite flow property is equivalent to:
  \begin{equation}
    \sum_{n=0}^{\infty} \sum_{i \in \bar{T}} \sum_{j \in T} a_{ij}(n) = \infty
    \label{inf-eq}
  \end{equation}
  for any subset $T$ of $S$ as in Eq. (\ref{infinite-flow}).
  Similarly, the absolute infinite flow property is equivalent to:
  \begin{equation}
    \sum_{n=0}^{\infty} \sum_{i \in \bar{T}(n+1)} \sum_{j \in T(n)} a_{ij}(n) = \infty
    \label{abs-eq}
  \end{equation}
  for any sequence $T(n)$ of subsets of $T$ as in Eq. (\ref{abs}). This can be easily seen by combining Eqs. (\ref{abs}) and (\ref{b-asym}).
\end{remark}

\subsection{Self-Confidence, Type-Symmetry, and\\ subsymmetry}

\begin{definition}
  A chain $(A_n)$ of stochastic matrices is said to have the \textit{self-confidence} property, or \textit{self-confident}, if there exists $\delta > 0$ such that for every $i$, $a_{ii} > \delta$.
\end{definition}
It is worth noting that if a chain is self-confident, then the absolute infinite flow property is equivalent to the infinite flow property. We also note that for self-confident chains, the balanced asymmetry property becomes equivalent to a simpler property that is defined in the following.
\begin{definition}\cite{Hend:11}
  A chain of row stochastic matrices is \textit{type-symmetric} if there exists a finite $M > 0$ such that for every $T \subset S = \{ 1,\ldots,s \}$, excluding the empty set and $S$,
  \begin{equation}
    \sum_{i \in T} \sum_{j \in \bar{T}} a_{ij}(n) \leq M \sum_{i \in \bar{T}} \sum_{j \in T} a_{ij}(n), \,\,\,\,\,\forall n \geq 0.
    \label{type-sym}
  \end{equation}
\end{definition}

In the following, according to \cite{Bolouki:11}, we define a special case of type-symmetry.

\begin{definition}
  A chain $(A_n)$ of row stochastic matrices is called \textit{sub-symmetric} if there exists a finite $M > 0$ such that for every two agents $i$ and $j$, we have
  \begin{equation}
    a_{ij}(n) \leq M a_{ji}(n), \,\,\, \forall n \geq 0.
  \end{equation}
\end{definition}

\subsection{Strong Interaction Digraph of A Chain}
\label{persistency}
The strong interaction digraph of a chain is of importance in class-ergodicity analysis. To evaluate the total influence of an agent on another agent, we define a function of ordered pairs of agents, as follows.
\begin{definition}
  Let $(A_n)$ be the chain representing interaction rates between $s$ agents. We define function $int: S \times S \rightarrow \mathbf{R}^{\geq 0}$ by
  \begin{equation}
    int(i,j) = \sum_{n=0}^{\infty} a_{ij}(n).
  \end{equation}
  An ordered pair (i,j) of agents is said to be \textit{highly interactive}, if
  \begin{equation}
    int(i,j) = \sum_{n=0}^{\infty} a_{ij}(n) = \infty.
  \end{equation}
\end{definition}
Note that from $(i,j)$ being highly interactive, one cannot conclude that $(j,i)$ is highly interactive.

For a chain $(A_n)$, strong interaction digraph $G_A$ is formed as follows. Considering $(A_n)$ as interactions between $s$ agents, $G_A$ contains $s$ nodes, each representing an agent. An edge is drawn from node $i$ to node $j$ if and only if the ordered pair $(i,j)$ is highly interactive.

One defines a relation $R$ between the $s$ nodes of $G_A$ (agents) as follows. For every two nodes $i$ and $j$ that are not necessarily distinct, $i \, R \, j$ if and only if there is a directed path from $i$ to $j$.

The following lemma, as stated in \cite{Bolouki:12}, shows how strong interaction digraph is related to balanced asymmetry.
\begin{lemma}
   Assume that $(A_n)$ is a balanced asymmetric chain, and $G_A$ is its strong interaction digraph. Then, $R$ is an equivalence relation
\end{lemma}
\begin{note}
  The equivalence relation $R$ partitions agents into equivalence classes, herein called \textit{islands}.
\end{note}


\section{Main Theorems}

In this section, we state our main results in two parts: ergodicity results and class-ergodicity results.

\subsection{Ergodicity Results}
The following theorem is established in \cite{Bolouki:12}.
\begin{theorem}
  If chain $(A_n)$ is $l_1$-approximation of a balanced asymmetric chain, then $(A_n)$ is ergodic if and only if it has the absolute infinite flow property.
  \label{ergodic}
\end{theorem}

\subsection{Class-Ergodicity Results}

From \cite{Bolouki:12}, we have the following theorem on class-ergodicity of chains.
\begin{theorem}
  Let chain $(A_n)$ be $l_1$-approximation of a balanced asymmetric chain. Then, chain $(A_n)$ is class-ergodic if and only if the absolute infinite flow property holds over each island of the strong interaction digraph induced by $(A_n)$.
  \label{result-semi}
\end{theorem}

In the following, by employing Theorem \ref{result-semi}, we prove a sufficient condition for a chain to be class-ergodic.
\begin{theorem}
  If chain $(A_n)$ is $l_1$-approximation of a self-confident and type-symmetric chain, it is also class-ergodic.
  \label{self-type}
\end{theorem}
\begin{proof}
  From Lemma \ref{partial ergodicity}, to prove class-ergodicity of $(A_n)$, we can assume that $(A_n)$ is self-confident and type-symmetric. These two properties of $(A_n)$ imply that $(A_n)$ is balanced asymmetric. Therefore, according to Theorem \ref{result-semi}, it suffices to show that the absolute infinite flow property holds over each island of the strong interaction digraph $G_A$. Let $I$ be an arbitrary island and $T(0),T(1),\ldots$ be an arbitrary sequence of subsets of $I$ with the same cardinality. Calling the equivalence of Eqs. (\ref{abs-eq}) and (\ref{abs}) for chains that are $l_1$-approximation of balanced asymmetric chains, in the following, we show that Eq. (\ref{abs-eq}) holds by considering the following two cases:

  Case I. The sequence $T(0),T(1),\ldots$ becomes constant after a finite time, i.e., there exist $T \subset I$ and $N \geq 0$ such that $T(i)=T$ for every $n \geq N$. In this case,
  \begin{equation}
    \sum_{n=0}^{\infty} \sum_{i \in \bar{T}(n+1)} \sum_{j \in T(n)} a_{ij}(n) \geq \sum_{n=N}^{\infty} \sum_{i \in \bar{T}} \sum_{j \in T} a_{ij}(n)
    \label{ya khoda}
  \end{equation}
  Clearly, there exist two agents $p \in \bar{T}$ and $q \in T$ such that pair $(p,q)$ is highly interactive. Otherwise, island $I$ would not be strongly connected. Noting that the RHS of ($\ref{ya khoda}$) is not less than
  \begin{equation}
    \sum_{n=N}^{\infty} a_{pq}(n)
  \end{equation}
  which diverges, since $(p,q)$ is highly interactive, the result if proved.

  Case II. The sequence $T(0),T(1),\ldots$ does not converge, i.e., there exists a time subsequence $n_0,n_1,\ldots$ such that $T(n_k) \neq T(n_k+1)$ for every $k=0,1,\ldots$. Clearly,
  \begin{equation}
    \begin{array}{ll}
      \sum_{n=0}^{\infty} \sum_{i \in \bar{T}(n+1)} \sum_{j \in T(n)} a_{ij}(n) \vspace{.05in}\\ \hspace{.5in}\geq \sum_{k=0}^{\infty} \sum_{i \in \bar{T}(n_k+1)} \sum_{j \in T(n_k)} a_{ij}(n_k)
    \label{ya khodaa}
    \end{array}
  \end{equation}
  Since $T(n_k) \neq T(n_k+1)$ and the two subsets are of the same cardinality, there exists an agent that belongs to both $\bar{T}(n_k+1)$ and $T(n_k)$. Hence, due to self-confidence of chain $(A_n)$, we have
  \begin{equation}
   \sum_{i \in \bar{T}(n_k+1)} \sum_{j \in T(n_k)} a_{ij}(n_k) > \delta.
  \end{equation}
  Therefore, the RHS of Eq. (\ref{ya khodaa}) is not less than $\sum_{k=0}^{\infty} \delta$ that diverges. This proves the result again in this case.
\end{proof}


\section{Relationship to known models}

In this section, we apply our main theorems to chains corresponding to different types of models and consensus algorithms found in the literature in order to analyze when their transition chains become ergodic or class-ergodic.

\subsection{JLM Model}

The parameter considered in \cite{Jadba:03} is the heading of each agent. If we write $\theta_i(n)$ as the heading of agent $i$ at moment $n$, the model describing evolution of headings is
\begin{equation}
  \theta_i(n+1)= \frac{1}{1+d_i(n)} [ \theta_i(n)+\sum_{j \in D_i(n)} \theta_j(n) ],
  \label{JLM}
\end{equation}
where $D_i(n)$ and $d_i(n)$ denote respectively the set and the number of neighbors of the agent $i$ at time $n$. It is also assumed that for each $n \geq 0$, if $j \in D_i(n)$, then $i \in D_j(n)$ too (undirected graph).

In \cite{Jadba:03}, the authors proved that a sufficient condition for consensus to occur is existence of an infinite sequence of contiguous, nonempty, bounded, time-intervals $[n_i,n_{i+1})$, $i=0,1,\ldots$, such that across each such interval, the $s$ agents are linked together.

In the following, we wish to apply Theorems \ref{ergodic} and \ref{self-type} to the transition chain of the JLM model. To take advantage of Theorem \ref{self-type}, we show that in the JLM model, the transition chain is both self-confident and type-symmetric. Note that
\begin{equation}
 a_{ii}(n) = 1/(1+d_i(n)) \geq 1/s
\end{equation}
This proves the self-confidence of the chain. To prove the type-symmetry, it suffices to show that the chain is sub-symmetric. If at time $n \geq 0$, we have $j \not\in D_i(n)$, then $i \not\in D_j(n)$ either. Therefore $a_{ij}(n) = a_{ji}(n) = 0$ (consistently with the subsymmetry requirement). If $j \in D_i(n)$ then $i \in D_j(n)$ also. In this case, it is easy to see that $a_{ij}(n)$ and $a_{ji}(n)$ both lie in the interval $[1/2, 1/s]$. Therefore, the subsymmetry condition holds by setting $M = s/2$. Thus, from Theorem \ref{self-type}, we conclude that the chain is class-ergodic. In other words, in the JLM model, unconditional multiple consensus occurs without any additional assumption.

We also note that the chain is balanced asymmetric as well, since self-confidence and type-symmetry imply balanced asymmetry. Thus, from Theorem \ref{ergodic}, we obtain that the absolute infinite flow property is equivalent to the ergodicity of the chain. On the other hand, since the chain is self-confident, the absolute infinite flow property is equivalent to the infinite flow property. Hence, in the JLM model, the infinite flow property is necessary and sufficient for ergodicity of the transition chain.

Graph interpretation of the infinite flow property in the JLM model is as follows. Due to the subsymmetry property of the JLM model, if a pair $(i,j)$ is highly interactive, then $(j,i)$ is highly interactive also. Therefore, in this case, the strong interactions digraph can be considered as an undirected graph. The infinite flow property is now equivalent to connectivity of the strong interactions graph.

 Another equivalent condition to the infinite flow property, that is more similar to the condition derived in \cite{Jadba:03}, is existence of an infinite sequence of contiguous and non empty time-intervals $[n_i,n_{i+1})$, $i \geq 0$, with the property that across each such interval, the $s$ agents are linked together. Note that the boundedness of the time-intervals is not required unlike in \cite{Jadba:03}. More importantly, the condition derived here is not only sufficient, but also necessary for ergodicity of the chain (unconditional occurrence of consensus in the model). On the other hand, unlike \cite{Jadba:03}, without extra conditions, no statement can be made about the speed of convergence to consensus.


\subsection{Models with Finite Range Interactions}

The Krause model \cite{Krause:97} is an example of endogenous models with finite range interactions. These models are special cases of first order models in which interaction rates depend directly on states. In these models, agent $i$ receives information from agent $j$ if and only if the distance between the two agents is less than some pre-specified level $R_i$, which is in general different for distinct agents. In the following, we define the interactions rates between agents. For every agent $i$, we set a decaying function $f_i : R^{\geq 0} \rightarrow R^{\geq 0}$ that vanishes at $R_i$, and define
\begin{equation}
 a_{ij} = \frac{f_i(\|X_i-X_j\|)}{\sum_{k=1}^s f_i(|X_i-X_k|)}
 \label{hoof}
\end{equation}
We consider a particular case of models defined above. Assume that agents have the same range of connectivity, i.e., $R_i = R$ for every $i$,
and use identical decaying functions, i.e., $f_i = f$. In the Krause model, $f(x) = 1$ for $0 \leq x < R$ and $f(x) = 0$ elsewhere.

It can be proved that in this case, the transition chain is self-confident with $\delta = 1/s$. It can also be shown as follows that the transition chain is sub-symmetric. If at time $n \geq 0$, agents $i$ and $j$ do not communicate, then $a_{ij}(n)=a_{ji}(n)=0$. If the two agents communicate, then $f(|x_i(n)-x_j(n)|) > 0$. Using Eq. (\ref{hoof}), we have
\begin{equation}
  \frac{a_{ij}(n)}{a_{ji}(n)} = \frac{\sum_{k=1}^s f(|X_j(n)-X_k(n)|)}{\sum_{k=1}^s f(|X_i(n)-X_k(n)|)}
  \label{fin-ran}
\end{equation}
Noting that $f$ is non increasing and $f(|X_i(n)-X_i(n)|)= f(|X_j(n)-X_j(n)|) = f(0)$, we conclude that the RHS of Eq. (\ref{fin-ran}) lies in interval $[1/s,s]$. Hence, subsymmetry is established by setting $M=s$. The chain being both self-confident and sub-symmetric, it is also type-symmetric. Thus, according to Theorem \ref{self-type}, the chain is class-ergodic, i.e., unconditional multiple consensus occurs.

\subsection{The C-S model}

The C-S (Cucker-Smale) model \cite{Cucker:07} is an example of endogenous consensus models with interaction rates remaining strictly positive. We apply our results to a generalized version of the C-S model \cite{Cucker:07} that describes evolution of positions $X_i$'s and velocities $V_i$'s in a bird flock, in a three dimensional Euclidian space:
\begin{equation}
 \begin{array}{ll}
  X_i(n+1) & \hspace{-2in} = X_i(n) + h V_i(n),\vspace{.05in} \\
  V_i(n+1) & \hspace{-2in} = V_i(n) + \sum_{j\neq i} f (\|X_i(n) - X_j(n) \|)\vspace{.05in} \\
  \hspace{1.6in}(V_j(n) - V_i(n)),
 \end{array}
 \label{cucker-smale model}
\end{equation}
where $f : \textbf{R}^{\geq 0} \rightarrow \textbf{R}^{\geq 0}$ is a non increasing function. Note that in this model, the limiting behavior of velocities is of interest. We have

\begin{equation}
  a_{ij}(n) = f (\|X_i(n) - X_j(n) \|),\,\, \forall i \neq j
\end{equation}
and
\begin{equation}
  a_{ii}(n) = 1- \sum_{j\neq i} f (\|X_i(n) - X_j(n) \|), \,\, \forall i
\end{equation}
Clearly, the transition chain in this algorithm is symmetric. To enforce self-confident, one may require an additional assumption, such as, $f(y) < 1/s$ for any $y \geq 0$. By this assumption, we have $a_{ii}(n) > 1/s$ for every $i=1,\ldots,s$ and $n \geq 0$. Therefore, the chain becomes self-confident. The combination of the self-confidence and type-symmetry of the chain allows an application of Theorem \ref{self-type} to yield the following result.
\begin{theorem}
  Consider the algorithm described by Eq. (\ref{cucker-smale model}). If $f(y)<1/s$ for any $y \geq 0$, then the transition chain is class-ergodic, and consequently, unconditional multiple consensus occurs.
\end{theorem}
To state the consensus result for the generalized C-S model, we define parameters $M_x$ and $M_v$ calculated from initial positions and velocities:
\begin{equation}
 M_x = \max_{i,j} \{ \| X_i(0) - X_j(0) \| | 1 \leq i < j \leq s \}
 \label{M_x}
\end{equation}
\begin{equation}
 M_v = \max_{i,j} \{ \| V_i(0) - V_j(0) \| | 1 \leq i < j \leq s \}.
 \label{M_v}
\end{equation}
\begin{theorem}
 For the multi-agent system with dynamics described by Eq. (\ref{cucker-smale model}), assume that $f(y)$ has the following property:
 \begin{equation}
  f(y) < 1/s, \,\, \forall s \geq 0
  \label{f-bound}
 \end{equation}
 Assume also that
 \begin{equation}
  M_v < \frac{s}{3 h} \int_{M_x}^{\infty} f(y) dy.
  \label{assumption-disc}
 \end{equation}
 Then, all agents' velocities converge to a common value. Moreover, there exists a non negative number $R$ such that for every $i,j$, $1 \leq i , j \leq s$,
 \begin{equation}
  \| X_i(n) - X_j(n) \| \leq R, \,\, \forall n \geq 0
 \end{equation}
 \label{cs_erg}
\end{theorem}
Unlike the models described previously, Theorem \ref{cs_erg} is not an immediate result of Theorems \ref{ergodic} and \ref{self-type}. However, to prove Theorem \ref{cs_erg}, we employ a technique similar to that used in the proof of Theorem \ref{ergodic} in \cite{Bolouki:12}.
\begin{proof}
 For every $i=1,\ldots,s$, let $V_{i1}(n), V_{i2}(n), V_{i3}(n)$ be components of $V_i(n)$, i.e.,
 \begin{equation}
  V_i(n) = [V_{i1}(n) \,\, V_{i2}(n) \,\, V_{i3}(n)]^T
 \end{equation}
 It is straight forward to verify that $V_{ir}$'s ($r=1,2,3$) satisfy for straightforwardly identifiable coefficients $a_{ij}$, the same update equation as $V_i$'s do, i.e.,
 \begin{equation}
  V_{ir}(n+1) - V_{ir}(n) = \sum_{j \neq i} a_{ij}(n) (V_{jr}(n) - V_{ir}(n))
  \label{2}
 \end{equation}
 One can rewrite Eq. (\ref{2}) as the following:
 \begin{equation}
   \begin{array}{ll}
    V_{ir}(n+1) & \hspace{-.1in} = \left( 1-\sum_{j \neq i} a_{ij}(n) \right) V_{ir}(n)\vspace{.05in} \\
    & + \sum_{j \neq i} a_{ij}(n) V_{jr}(n)
  \label{convex}
  \end{array}
 \end{equation}
 Let us define $z_{ir}(n):\mathbf{R^{\geq 0}} \rightarrow \mathbf{R}$, $1 \leq i \leq s$, $r = 1,2,3$ from $V_{ir}(n)$'s as follows. At every time $n \geq 0$, $z_{ir}(n)$ is equal to the
 $i$th least number among $V_{1r}(n),\ldots,V_{sr}(n)$.

 Note that the coefficients in the RHS of Eq. (\ref{convex}) are all positive and add up to 1. This means that $V_{ir}(n+1)$ is a convex combination of the values $V_{1r}(n),\ldots,V_{sr}(n)$. Thus, for every $r$, $1 \leq r \leq 3$, interval $[z_{1r}(n),z_{sr}(n)]$, which is the smallest interval containing all the values $V_{ir}(n)$'s, shrinks during time. Particularly, this shows that $V_{ir}(n)$, and consequently $z_{ir}(n)$, are bounded for every $i=1,\ldots,s$ and $r=1,2,3$.
 For consensus to occur, we require that interval $[z_{1r}(n),z_{sr}(n)]$ converges to a point for every $r$, $1 \leq r \leq 3$. In the following, we investigate how these intervals shrink with time. Let us define:
 \begin{equation}
  z(n) = \sum_{r = 1}^3 (z_{sr}(n) - z_{1r}(n)).
 \end{equation}
 We know that
 \begin{equation}
   \begin{array}{ll}
     X_i(n) - X_j(n) & \hspace{-.1in}= X_i(0) - X_j(0)\vspace{.05in}\\
     & + h \sum_{m=0}^{n-1} (V_i(m) - V_j(m))
   \end{array}
 \end{equation}
 Thus,
 \begin{equation}
   \begin{array}{ll}
    \|X_i(n) - X_j(n)\| & \hspace{-.1in}\leq \|X_i(0) - X_j(0)\|\vspace{.05in}\\
    &+ h \sum_{m=0}^{n-1} \|V_i(m) - V_j(m)\|
  \label{4}
    \end{array}
 \end{equation}
 However,
 \begin{equation}
  \begin{array}{ll}
   \|V_i(m) - V_j(m)\| & \hspace{-.1in}\leq \sum_{r=1}^3 |V_{ir}(m) - V_{jr}(m)| \vspace{.05in} \\
                             & \hspace{-.1in}\leq \sum_{r=1}^3 (z_{sr}(m) - z_{1r}(m)) \vspace{.05in} \\
                             & \hspace{-.1in}\equiv z(m)
  \end{array}
  \label{5}
 \end{equation}
 Eqs. (\ref{M_x}), (\ref{4}) and (\ref{5}) imply
 \begin{equation}
  \|X_i(n) - X_j(n)\| \leq M_x + h \sum_{m=0}^{n-1} z(m)
  \label{18}
 \end{equation}
 From $f$ being non increasing, we obtain
 \begin{equation}
  f(\|X_i(n) - X_j(n)\|) \geq f(M_x + h \sum_{m=0}^{n-1} z(m))
  \label{low}
 \end{equation}
 Note that the RHS of Eq. (\ref{low}) above is independent of $i,j$. Defining
 \begin{equation}
  g(n) = M_x + h \sum_{m=0}^{n-1} z(m),
 \end{equation}
 Eq. (\ref{low}) implies
 \begin{equation}
  a_{ij}(n) \geq f(g(n)), \,\, \forall i,j, \,\, \forall n \geq 0
  \label{9'}
 \end{equation}
 We know recall Eq. (\ref{convex}). It is straight forward to verify that all the coefficients in the RHS of Eq. (\ref{convex}) lie between $f(g(n))$ and $1-(s-1)f(g(n))$. Thus, as the sum of coefficients is 1, to find a lower bound for the value of $V_{ir}(n+1)$, we put higher weights on lower valued $V_{jr}$'s, in particular, $1-(s-1)f(g(n))$ on the least one, which is $z_{1r}(n)$ and $f(g(n))$ on the rest of them. Hence, we conclude
 \begin{equation}
   \begin{array}{ll}
     V_{ir}(n+1) & \hspace{-.1in} \geq \Big(1-(s-1)f(g(n))\Big)z_{1r}(n)\vspace{.05in}\\
      &+ \sum_{j=2}^s f(g(n))z_{jr}(n)
   \end{array}
 \end{equation}
 As a result,
 \begin{equation}
   \begin{array}{ll}
  z_{1r}(n) & \hspace{-.1in} \geq \Big(1-(s-1)f(g(n)) \Big) z_{1r}(n)\vspace{.05in}\\
  &+ \sum_{j=2}^s f(g(n))z_{jr}(n)
  \label{10}
   \end{array}
 \end{equation}
 Using the opposite process to build an upper bound for the value of $V_{ir}(n+1)$, we obtain:
 \begin{equation}
   \begin{array}{ll}
  z_{sr}(n) & \hspace{-.1in} \leq \Big(1-(s-1)f(g(n)) \Big) z_{sr}(n)\vspace{.05in}\\
  & + \sum_{j=1}^{s-1} f(g(n))z_{jr}(n)
  \label{11}
   \end{array}
 \end{equation}
 Subtracting Eq. (\ref{10}) from Eq. (\ref{11}) implies
 \begin{equation}
   \begin{array}{ll}
  z_{sr}(n+1) - z_{1r}(n+1)& \hspace{-.1in} \leq
   \vspace{.05in}\\
  &\hspace{-1in}\Big(1-s f(g(n))\Big) \Big(z_{sr}(n) - z_{1r}(n)\Big)
  \label{12}
   \end{array}
 \end{equation}
 By adding up Eq. (\ref{12}) for $r=1,2,3$, we obtain
 \begin{equation}
  z(n+1) \leq \Big(1-s f(g(n))\Big) z(n)
  \label{13}
 \end{equation}
 or equivalently,
 \begin{equation}
  z(n+1) - z(n) \leq -sf(g(n)) z(n)
  \label{egh}
 \end{equation}
 We now note that $h z(n) = g(n+1) - g(n)$. Since $f$ is non increasing and $g(n+1) - g(n) = h z(n) \geq 0$, we have
 \begin{equation}
   \begin{array}{ll}
     f(g(n)) z(n) & \hspace{-.1in} = f(g(n)) \frac{g(n+1)-g(n)}{h} \vspace{.05in} \\
     & \hspace{-.1in}\geq \frac{1}{h} \int_{g(n)}^{g(n+1)}f(y)dy
  \label{eghz}
   \end{array}
 \end{equation}
 Eqs. (\ref{egh}) and (\ref{eghz}) imply
 \begin{equation}
  z(n+1) - z(n) \leq \frac{-s}{h} \int_{g(n)}^{g(n+1)} f(y) dy
  \label{eghzz}
 \end{equation}
 The above equation holds for every $n \geq 0$. If we substitute variable $n$ in Eq. (\ref{eghzz}) with $n'$ and sum it up for $n' = 1,\ldots,n-1$. We obtain
 \begin{equation}
  z(n) - z(0) \leq \frac{-s}{h} \int_{g(0)}^{g(n)} f(y) dy
 \end{equation}
 Recalling the definition of $g(n)$ we conclude
 \begin{equation}
  z(n) - z(0) \leq \frac{-s}{h} \int_{M_x}^{M_x + h \sum_{m=0}^{n-1} z(m)} f(y) dy
  \label{ajab}
 \end{equation}
 If consensus does not occur, then $\sum_{m=0}^{\infty} z(n)$ diverges. Thus by taking $n$ to infinity and noting that $\lim_{n \rightarrow \infty} z(n)$ exists as $z(n)$ is non increasing and non negative, if consensus were to fail, Eq. (\ref{ajab}) would be modified as
 \begin{equation}
  \lim_{n \rightarrow \infty}z(n) - z(0) \leq \frac{-s}{h} \int_{M_x}^{\infty} f(y) dy
  \label{ajabz}
 \end{equation}
 On the other hand, according to our assumption Eq. (\ref{assumption-disc}), we know that
 \begin{equation}
  \begin{array}{ll}
   z(0) & \hspace{-.1in} = \sum_{r=1}^3 (z_{sr}(0)-z_{1r}(0)) \vspace{.05in}\\
        & \hspace{-.1in} = \sum_{r=1}^3 \max_{i,j} \{ (V_{ir}(0) - V_{jr}(0) \} \vspace{.05in} \\
        & \hspace{-.1in} \leq \sum_{r=1}^3 \max_{i,j} \{ (\| V_{i}(0) - V_{j}(0) \| \} = 3 M_v \vspace{.05in}\\
        & \hspace{-.1in} < \frac{s}{h} \int_{M_x}^{\infty} f(y) dy
  \end{array}
  \label{ajabzz}
 \end{equation}
 Eqs. (\ref{ajabz}) and (\ref{ajabzz}) together imply that
 \begin{equation}
  \lim_{n \rightarrow \infty}z(n) < 0
 \end{equation}
 which is a contradiction, since $z(n)$ is non negative. Hence, consensus must occur and moreover, $\sum_{n=1}^{\infty} z(n)$ converges. Recalling Eq. (\ref{18}) we conclude that $\|X_i(n) - X_j(n)\|$ is  bounded for every $i,j$, i.e., there is $R \geq 0$ such that
 \begin{equation}
  \| X_i(n) - X_j(n) \| \leq R, \,\, \forall i,j, \,\, \forall n \geq 0
  \label{100}
 \end{equation}

\end{proof}

Applying Theorem \ref{cs_erg} to the C-S model with
\begin{equation}
  f(y) = \frac{K}{(\sigma^2 + y^2)^{\beta}}
  \label{ine}
\end{equation}
we obtain the following.
\begin{corollary}
 Let the dynamics of a multi agent system be described by Eq. (\ref{cucker-smale model}) with $f$ defined in Eq. (\ref{ine}). Assume that $K / \sigma^{2\beta} < 1/s$. Then, under any of the following conditions, agents velocities converge to a common value:
  \begin{enumerate}
   \item $\beta \leq 1/2$
   \item $\beta > 1/2$ and
    \begin{equation}
     M_v < \frac{sK}{3h(2\beta -1)(M_x + \sigma)^{2\beta -1}}
     \label{my-cucker-1}
    \end{equation}
  \end{enumerate}
 \end{corollary}


\section{Concluding Remarks}

In this paper, we analyzed the limiting behavior of the JLM model, the Cucker-Smale model, and the Krause model, using our general theorems about consensus and multiple consensus. In Theorem \ref{ergodic}, we presented a necessary and sufficient condition for chains that are $l_1$-approximation of balanced asymmetric chains to be ergodic. Noting that the transition chain in the JLM model is balanced asymmetric, we found a necessary and sufficient condition for unconditional consensus to occur in the JLM model. Theorem \ref{result-semi} provides a necessary and sufficient condition for class-ergodicity of those chains that are $l_1$-approximation of balanced asymmetric chains. Theorem \ref{self-type} is a non trivial especial case of Theorem \ref{result-semi}, that contains a sufficient condition for class-ergodicity. Theorems \ref{result-semi} and \ref{self-type} led us to prove unconditional multiple consensus in various known models: the JLM model, finite range interactions models, the Krause model for instance, and the generalized version of the C-S model. Furthermore, we obtained a sufficient condition for consensus in the C-S model by employing elementary methods.

In future work, we wish to investigate the rate of convergence when consensus or multiple consensus occurs in a system. In particular, we are interested in necessary and sufficient conditions for exponentially fast convergence. We also wish to extend our results to random networks as well as systems with large population of agents.


%

\end{document}